\documentclass[11pt,a4paper]{article}
\usepackage[greek,english]{babel}
\usepackage[iso-8859-7]{inputenc}
\usepackage[inline]{enumitem}
\usepackage{amsthm,amsmath,amssymb,graphicx,bbm,mdframed,extarrows,float,tikz,changepage,subcaption,ifthen,chngcntr}
\usetikzlibrary{decorations.pathreplacing}
\usetikzlibrary{arrows.meta}
\usetikzlibrary{patterns}
\usetikzlibrary{intersections}

\makeatletter
\newcommand{\leqnomode}{\tagsleft@true}
\newcommand{\reqnomode}{\tagsleft@false}
\makeatother

\allowdisplaybreaks

\newtheoremstyle{myStyle}
{\topsep}
{\topsep}
{}
{}
{\bfseries}
{.}
{5pt}
{}

\theoremstyle{myStyle}

\newtheorem{myDef}{Definition}[section]
\newtheorem{thm}[myDef]{Theorem}
\newtheorem*{unthm}{Theorem}
\newtheorem{cor}[myDef]{Corollary}

\newtheorem{rmrk}[myDef]{Remark}
\newtheorem{prop}[myDef]{Proposition}

\newcommand{\norm}[1]{\left\lVert#1\right\rVert}
\newcommand{\pn}[1]{\left(#1\right)}
\newcommand{\N}{\mathbb{N}}
\newcommand{\normed}{(X,\norm{\cdot})}
\newcommand{\dnormed}{(X^*,\norm{\cdot})}

\DeclareMathOperator{\conv}{conv}
\DeclareMathOperator{\ant}{ant}

\numberwithin{equation}{section}

\title{\textbf{Antipodal Sets in infinite dimensional Banach spaces}}
\author{E. Glakousakis and S. Mercourakis}
\date{}

\begin{document}
\maketitle
\begin{abstract}
The following strengthening of the Elton-Odell theorem on the existence of a $(1+\epsilon)-$separated sequences in the unit sphere $S_X$ 
of an infinite dimensional Banach space $X$ is proved: There exists an infinite subset $S\subseteq S_X$ and a constant $d>1$, satisfying the 
property that for every $x,y\in S$ with $x\neq y$ there exists $f\in B_{X^*}$ such that $d\leq f(x)-f(y)$ and $f(y)\leq f(z)\leq f(x)$, for 
all $z\in S$.
\end{abstract}
\setcounter{section}{-1}
\section{Introduction} A set $S$ in a Banach space $X$ is called $d-$separated $(d>0)$ if $\norm{x-y}\geq d$ $\forall\ x\neq y\in S$. For 
infinite dimensional Banach spaces the parameter $K(X)=\sup\{d:\exists S\subseteq B_X, S\text{ infinite and $d-$separated}\}$ is called 
Kottman's constant or separation constant of $X$ and by a well known Theorem of Elton and Odell [8] is strictly greater than 1. In the present 
paper we study the parameter $K_a(X)$ for infinite dimensional Banach spaces, which was introduced on [17] as $\ant(X)$. The definition of 
$K_a(X)$ is based on the notion of bounded and separated antipodal sets [17]. Bounded and separated antipodal sets were introduced as a 
strengthening of the classical concept of antipodal sets (see [5] and [18]) to include spaces of any dimension whereas the original definition 
was suitable for spaces of finite dimension [18]. We remind the above definitions.
\begin{myDef}[{[18]}]
A subset of an $n-$dimensional real vector space $X$ is said to be antipodal if for every $x,y\in S$ with $x\neq y$ there exist distinct 
parallel support hyperplanes $P,Q$ such that $x\in P$ and $y\in Q$.
\end{myDef}
\begin{myDef}[{[17]}]
Let $(X,\norm{\cdot})$ be a normed space.
\begin{enumerate}[label=(\alph*)]
\item A subset $S$ of $X$ is called antipodal if for every $x,y\in S$ with $x\neq y$ there exists $f\in X^*$ such that $f(x)<f(y)$ and 
$f(x)\leq f(z)\leq f(y)$ $\forall z\in S$.
\item An antipodal subset $S$ of $X$ is said to be bounded and separated, in short b.s.a subset, if there exist positive constants $c_1,c_2$ 
and $d$ such that
\begin{enumerate}[label=(\arabic*)]
\item$\norm{x}\leq c_1$, $\forall x\in S$ and
\item for every $x,y\in S$ with $x\neq y$ there exists $f\in X^*$ with $\norm{f}\leq c_2$, such that $0<d\leq f(y)-f(x)$ and 
$f(x)\leq f(z)\leq f(y)$ $\forall z\in S$.
\end{enumerate}
A subset $S$ of a normed space $X$, as above, will be called $(c_1,c_2,d)-$b.s.a subset of $X$.
\item If $X$ is infinite dimensional we set \[K_a(X)=\sup\{d:\exists\ S\subseteq B_X,S\text{ infinite }(1,1,d)-\text{b.s.a set}\}.\]
\end{enumerate}
\end{myDef}
It is clear that if $(X,\norm{\cdot})$ is finite dimensional then Definition 0.1 and Definition 0.2 (a) coincide.

In relation with Elton-Odell Theorem it would be interesting to know if every infinite dimensional Banach space $X$ contains an infinite bounded and 
separated antipodal subset with constants $c_1=c_2=1$ and $d>1$ or equivalently if $K_a(X)>1$, for every infinite dimensional Banach space. 
Indeed it is obvious that in that case we would have a stronger version of Elton-Odell Theorem. The above question was posed in [17] and our 
main aim is to provide an affirmative answer. For spaces that contain isomorphically $c_0$ or $l_p$ for some $1\leq p<\infty$ suffice the 
structural properties of these spaces (Proposition 1.2) while for uniformly smooth spaces only the geometric properties of those spaces are needed 
(Proposition 1.4). The main tool in order to pass to more general classes of spaces is Theorem 1.7 which proof is essentially based on the proof 
of Theorem 1 of [17]. Using Theorem 1.7 we prove that $K_a(X)>1$ when $X$ is reflexive Banach space (Corollary 1.8) or $X$ has separable dual 
(Corollary 1.9). For the general case apart from Theorem 1.7 is also needed the highly non trivial Theorem 4.1 of [9]. If $X$ is any (real) 
Banach space then $B_X$ (resp. $S_X$) denotes its closed unit ball (resp. unit sphere). The Banach-Mazur distance between two isomorphic 
Banach spaces $X$ and $Y$ is defined as
\[d(X,Y)=\inf\{\norm{T}\norm{T^{-1}}:T\text{ is an invertible operator from $X$ onto $Y$}\}.\]
\section{Bounded and separated antipodal sets in infinite dimensions.}From here on we concern ourselves with infinite dimensional Banach spaces, 
except stated otherwise. We start this section with some remarks concerning bounded and separated antipodal sets (see [17]).
\begin{rmrk}
Let $(X,\norm{\cdot})$ be a Banach space.
\begin{enumerate}[label=(\arabic*)]
\item Let $S$ be a bounded and separated antipodal subset of $X$. It is easy to see that if $\lambda>0$, $S$ is also bounded and separated with 
constants $c_1,\lambda c_2,\lambda d$ and the 
same is valid for the set $\lambda S=\{\lambda x:x\in S\}$ with constants $\lambda c_1,c_2,\lambda d$. Thus a bounded and separated 
antipodal subset of $X$ can be defined as a subset $S$ of $B_X$ that satisfies the property: there exists $d>0$ such that for every $x\neq y\in S$
there is $f\in B_{X^*}$ with $d\leq f(y)-f(x)$ and $f(x)\leq f(z)\leq f(y)$, for every $z\in S$.
\item Let $S$ be a $\lambda-$equilateral ($\norm{x-y}=\lambda$ for every $x\neq y\in S$) subset of $X$. Then the set $S$ is a 
$(M,1,\lambda )-$b.s.a subset of $X$, where $M=\sup\{\norm{x}:x\in S\}$.
\item Let $\{(x_\gamma,x_\gamma^*):\gamma\in\Gamma\}$ be a bounded biorthogonal system in $X$ with $M\geq\norm{x_\gamma}\norm{x_\gamma^*}$, 
for every $\gamma\in\Gamma$. We consider the biorthogonal 
system $\{(y_\gamma,y_\gamma^*):\gamma\in\Gamma\}$ with $y_\gamma=\frac{x_\gamma}{\norm{x_\gamma}}$ and $y_\gamma^*=\norm{x_\gamma}{x_\gamma^*}$, $\gamma\in\Gamma$. 
Then the minimal system $\{y_\gamma:\gamma\in\Gamma\}$ is a $(1, M, 1)-$b.s.a subset of $X$.
\item The Elton-Odell Theorem states that: If $\dim X=\infty$, then there exists a $(1+\epsilon)-$
separated sequence in $S_X$. Therefore $K(X)>1$. Since it is apparent that $K(X)\leq2$ we get that $1<K(X)\leq2$.
\item Since every infinite dimensional Banach space $X$ contains an infinite Auerbach system, 
that is, a biorthogonal system $\{(x_n,x_n^*):n\in\mathbb{N}\}$ 
such that $\norm{x_n}=\norm{x_n^*}=1$, for $n\in\mathbb{N}$, (see [6] and [11] Th. 1.20) by (3) we get that $K_a(X)\geq1$. Also we have that 
$K_a(X)\leq K(X)$.
\item Let $Y$ be a subspace of $X$, then $K(Y)\leq K(X)$ and $K_a(Y)\leq K_a(X)$.
\end{enumerate}
\end{rmrk}
In the next Proposition, which strengthens Theorem 3 of [12], it is proved that if a Banach space $X$ contains isomorphicaly 
$c_0$ or $l_p$ for some $1\leq p\leq\infty$, then $K_a(X)>1$.

\begin{prop}
Let $X$ be a Banach space. If $X$ contains isomorphicaly $c_0$ or $l_1$, then $K_a(X)=2$. If $X$ contains $l_p$ for some $1<p<\infty$, then 
$K_a(X)\geq2^{1/p}$.
\end{prop}
\begin{proof}
We start by observing that the set $\left\{\sum_{k=1}^ne_k-e_{n+1}:n\in\mathbb{N}\right\}$ is normalized and 
2-equilateral in $c_0$ and the set $\{e_n:n\in\N\}$ is normalized and $2^{1/p}-$equilateral in $l_p$, $1\leq p<\infty$. 
We may assume that $(X,\norm{\cdot})\equiv c_0$ or $l_p$, $1\leq p<\infty$ and it suffices to prove the 
conclusion for an equivalent norm $\norm{\cdot}'$ on $X$. We tackle separately the cases $\normed\equiv c_0$ and $\normed\equiv l_p$, $1<p<\infty$. 
The case $\normed\equiv l_1$ can be proved with any of the two 
ways that will be presented, so it is excluded. Our proof is based essentially 
on the proof of James non-distortion Theorem and the remarks that follow its 
proof in [14] Prop. 2.e.3 from where we take that: Let $\normed\equiv c_0$ or $l_p$, $1\leq p<\infty$ and $\norm{\cdot}'$ an equivalent norm on $X$. 
Then for every $\epsilon>0$ there exists a block basic sequence $(x_k)$ of the canonical basis of $X$, $(e_n)$, such that
\begin{gather*}
\norm{x_k}'=1,\ k\in\N\text{ and}\\
\norm{\sum_{k=1}^\infty a_kx_k}'\geq\frac{1}{1+\epsilon}\norm{(a_k)}\text{ for every }(a_k)\subseteq X.
\end{gather*}
Moreover, if $\normed\equiv c_0$ or $l_1$ we have that
\[(1+\epsilon)\norm{(a_k)}\geq\norm{\sum_{k=1}^\infty a_kx_k}'\text{, for every }(a_k)\subseteq X.\]
We note that the space $Z=[x_k]$ is isomorphic to $X$ since the bases $(x_k)$ and $(e_k)$ are equivalent. 
With $(x_k^*)$ we denote the biorthogonal functionals of $(x_k)$.

Let now $\normed\equiv l_p$, $1<p<\infty$, $\epsilon>0$ and $(x_k)$ a block basic sequence of $(e_n)$ such that
\begin{gather*}
\norm{x_k}'=1,\ k\in\N\text{ and}\\
\norm{\sum_{k=1}^\infty a_kx_k}'\geq\frac{1}{1+\epsilon}\left(\sum_{k=1}^\infty|a_k|^p\right)^{1/p}\text{, for every }(a_k)\subseteq X.
\end{gather*}
We will show that the set $S=\{x_k:k\in\N\}$ is a $(1,1+\epsilon,2^{1/p})-$b.s.a. subset of $Z$. Indeed, we have that,
\begin{align*}
\norm{x_k^*}'&=\sup\left\{|a_k|:\norm{\sum_{n=1}^\infty a_nx_n}'\leq1\right\}\\
&\leq\sup\left\{|a_k|:\left(\sum_{n=1}^\infty|a_n|^p\right)^{1/p}\leq1+\epsilon\right\},\ k\in\N
\end{align*}
and for $k,l\in\N$ with $k\neq l$,
\begin{align*}
\norm{x_k^*-x_l^*}'&=\sup\left\{|a_k-a_l|:\norm{\sum_{n=1}^\infty a_nx_n}'\leq1\right\}\\
&\leq\sup\left\{|a_k-a_l|:\left(\sum_{n=1}^\infty|a_n|^p\right)^{1/p}\leq1+\epsilon\right\}\\
&=(1+\epsilon)2^{1/q}\text{, where }\frac{1}{p}+\frac{1}{q}=1.
\end{align*}
For $k,l\in\N$ with $k\neq l$ we set $g_{kl}=2^{(1/p)-1}(x_k^*-x_l^*)\in Z^*$, then
\begin{gather*}
\norm{g_{kl}}'=2^{(1/p)-1}\norm{x_k^*-x_l^*}\leq2^{(1/p)-1}(1+\epsilon)2^{1/q}=(1+\epsilon)\text{ and}\\
g_{kl}(x_k)=2^{(1/p)-1},\ g_{kl}(x_m)=0,\ g_{kl}(x_l)=-2^{(1/p)-1},\ m\not\in\{k,l\}.
\end{gather*}
Thus $S$ is a bounded and separated antipodal subset of $Z$ with constants $c_1=1$, $c_2=1+\epsilon$, $d=g_{kl}(x_k)-g_{kl}(x_l)=2\cdot2^{1/p-1}=2^{1/p}$.

From (1) of Remark 1.1 now it is direct that $S$ is a $\left(1,1,\frac{2^{1/p}}{1+\epsilon}\right)-$b.s.a. subset of $Z$, so we get that 
$K_a(X)\geq 2^{1/p}$.

For the case $\normed\equiv c_0$ we consider $\epsilon>0$ and $(x_k)$ a block basic sequence of $(e_n)$ such that
\begin{gather*}
\norm{x_k}'=1,\ k\in\N\text{ and}\\
\begin{align*}
&(1+\epsilon)\norm{(a_k)}_\infty\geq\norm{\sum_{k=1}^\infty a_kx_k}'\geq\frac{1}{1+\epsilon}\norm{(a_k)}_\infty\text{, for every choice}\\ 
&\text{of scalars, }(a_k)\text{, tending to zero.}
\end{align*}
\end{gather*}
For $n\in\N$,
\begin{gather*}
\norm{\sum_{k=1}^n x_k-x_{n+1}}'\leq(1+\epsilon)\norm{\sum_{k=1}^ne_k-e_{n+1}}=(1+\epsilon)\text{ and}\\
\begin{align*}
\norm{x_n^*}'&=\sup\left\{|a_n|:\norm{\sum_{k=1}^\infty a_kx_k}'\leq1\right\}\\
&\leq\sup\left\{|a_n|:\norm{\sum_{k=1}^\infty a_kx_k}\leq1+\epsilon\right\}\leq(1+\epsilon).
\end{align*}
\end{gather*}
We set $y_n=\frac{1}{1+\epsilon}\left(\sum_{k=1}^nx_k-x_{n+1}\right)$
and $y_n^*=\frac{1}{1+\epsilon}x_n^*$. We will show that 
the set $S=\left\{y_n:n\in\N\right\}$ 
is a $\left(1,1,\frac{2}{(1+\epsilon)^2}\right)-$b.s.a subset of $(X,\norm{\cdot}')$. Indeed for $m<n$ and $k\in\N$ it is easy to see that
\[-\frac{1}{(1+\epsilon)^2}=y_{m+1}^*(y_m)\leq y_{m+1}^*(y_k)\leq y_{m+1}^*(y_n)=\frac{1}{(1+\epsilon)^2}.\]
Since $\epsilon>0$ was arbitrary we get the coclusion.
\end{proof}
\begin{rmrk}
\begin{enumerate}[label=(\arabic*)]
\item From a classical result we know that $K(l_p)=2^{1/p}$, $1<p<\infty$, (see [4]). Thus $K_a(l_p)\leq K(l_p)=2^{1/p}$. On the other hand 
the canonical basis $\{e_n:n\in\N\}$ of $l_p$ is a 
normalized $2^{1/p}-$equilateral subset of $l_p$, so $K_a(l_p)\geq 2^{1/p}$ and consequently $K_a(l_p)=2^{1/p}$.
\item If $1<p<\infty$ and $X\cong l_p$ it is not valid that $K_a(X)=2^{1/p}$. There exists a renorming $\norm{\cdot}'$ of $l_2$ such that 
$K_a(l_2,\norm{\cdot}')\geq\sqrt{3}>\sqrt{2}$ (see [15]). 
Also in Proposition 1.13 it is proved that for any Banach space $X$ there exists an 
equivalent norm such that $K_a(X,\norm{\cdot}')=2$.
\end{enumerate}
\end{rmrk}
\begin{prop}
Let $X$ be a uniformly smooth Banach space, then $K_a(X)>1$.
\end{prop}
\begin{proof}
Let $\{(x_i,x_i^*):i\in\N\}$ be an Auerbach system in $X$. The space $X$ is uniformly smooth, so its dual 
space is uniformly convex. From the strict convexity of $X^*$ for $i\neq j$ and $s,t\in(0,1)$ with $s+t=1$ we have
\begin{equation}\label{eq:11}
\norm{sx_i^*-tx_j^*}<1.
\end{equation}
For $i\neq j$, we set $\lambda_{ij}=\frac{1}{\norm{\frac{1}{2}x_i^*-\frac{1}{2}x_j^*}}$. By \eqref{eq:11} $\lambda_{ij}>1$. Also
\[\norm{\frac{1}{2}x_i^*-\frac{1}{2}x_j^*}\geq\left(\frac{1}{2}x_i^*-\frac{1}{2}x_j^*\right)(x_i)=\frac{1}{2}.\]
So $\lambda_{ij}\in(1,2]$, $i\neq j$. We choose $M\in[\N]^\omega$ such that the following limit exists
\[\lim_{i<j\in M}\lambda_{ij}.\]
We will show that $\lim_{i<j}\lambda_{ij}>1$. Let now $M=(i_k)_{k\in\N}$. We assume that $\lim_{i<j\in M}\lambda_{ij}=1$. Then the sequences
\[\pn{\norm{\lambda_{i_ki_{k+1}}x_{i_k}^*}}_{k\in\N},\ \pn{\norm{\lambda_{i_ki_{k+1}}x_{i_{k+1}}^*}}_{k\in\N},\ 
\pn{\norm{\frac{1}{2}\lambda_{i_ki_{k+1}}(x_{i_k}^*-x_{i_{k+1}}^*)}}_{k\in\N}\]
all converge to $1$. By the uniform convexity of $X^*$ we get that
\[\norm{x_{i_k}^*+x_{i_{k+1}}^*}\to0,\ k\to\infty\ [16].\]
On the other hand $\pn{x_{i_k}^*+x_{i_{k+1}}^*}\pn{x_{i_k}}=1$, for $k\in\N$, 
a contradiction. Thus there exist $i_0\in M$ and $c>1$ such 
that $\lambda_{ij}\geq c$, for every $i_0\leq i<j\in M$. For $i_0\leq i<j\in M$, from the choice of $\lambda_{ij}$
we take that
\[\norm{\lambda_{ij}\pn{\frac{1}{2}x_i^*-\frac{1}{2}x_i^*}}=1.\]
Further
\[\lambda_{ij}\pn{\frac{1}{2}x_i^*-\frac{1}{2}x_j^*}(x_i-x_j)=\lambda_{ij}\pn{\frac{1}{2}+\frac{1}{2}}=\lambda_{ij}\geq c>1\]
and of course
\[\lambda_{ij}\pn{\frac{1}{2}x_i^*-\frac{1}{2}x_j^*}(x_k)=0,\ k\not\in\{i,j\}.\]
Our proof is complete.
\end{proof}
\begin{cor}
Let $\normed$ be a Βanach space such that $\dnormed$ is strictly convex (so $\normed$ is smooth) and $\{(x_i,x_i^*):i\in\N\}$ an Auerbach system 
in $X$. Then the set $\{x_i:i\in\N\}$ is a bounded and separated antipodal set 
with constants $c_1=c_2=1$ and $d=(1+)$ (that is for every $i\neq j\in\N\ \exists f\in B_{X^*}$ such that $1<f(x_i)-f(x_j)$ and 
$f(x_j)\leq f(x_k)\leq f(x_i)$, $k\in\N$).
\end{cor}
\begin{proof}
As in the proof of Proposition 1.4 from the strict convexity of $X^*$ for $i\neq j\in\N$ we set $\lambda_{ij}=\frac{1}{\norm{\frac{1}{2}x_i^*-\frac{1}{2}x_j^*}}$, 
$f=\lambda_{ij}\pn{\frac{1}{2}x_i^*-\frac{1}{2}x_j^*}$, and we get that $\lambda_{ij}>1$ and $\norm{f}=1$. Thus for $i\neq j\in\N$ and $k\in\N$ we have
\begin{gather*}
f(x_i)-f(x_j)=\lambda_{ij}\pn{\frac{1}{2}x_i^*-\frac{1}{2}x_j^*}(x_i-x_j)=\lambda_{ij}\pn{\frac{1}{2}+\frac{1}{2}}=\lambda_{ij}>1\text{ and}\\
-\frac{\lambda_{ij}}{2}=f(x_j)\leq f(x_k)\leq f(x_i)=\frac{\lambda_{ij}}{2}.
\end{gather*}
Our proof is complete.
\end{proof}
The next result concerns smooth Banach spaces of finite dimension.
\begin{prop}
Let $X$ be a smooth Banach space with $\dim X=n$. Then $X$ contains a $(1,1,d)-$b.s.a. set with $d>1$ and cardinality $2n$.
\end{prop}
\begin{proof}
Let $\{(x_i,x_i^*):1\leq i\leq n\}$ be an Auerbach basis of $X$. The space $X$ is smooth, consequently $X^*$ is strictly 
convex. Thus, since $\norm{\frac{1}{2}x_i^*\pm\frac{1}{2}x_j^*}<1$, for every $1\leq i\neq j\leq n$, there exist $a,b>1$ such that
\[\norm{a\pn{\frac{1}{2}x_i^*+\frac{1}{2}x_j^*}}=1\]
and
\[\norm{b\pn{\frac{1}{2}x_i^*-\frac{1}{2}x_j^*}}=1.\]
Equivalently for every $1\leq i\neq j\leq n$ there exist $s,t>\frac{1}{2}$ such that $\norm{s(x_i^*+x_j^*)}=1$ and $\norm{t(x_i^*-x_j^*)}=1$.

We will show that the set $\{\pm x_i:1\leq i\leq n\}$ satisfies the conclusion. Let now $1\leq i\neq j\leq n$, we set $\phi_1=t\pn{x_i^*-x_j^*}$ 
and $\phi_2=s\pn{x_i^*+x_j^*}$. Then
\leqnomode
\begin{align*}
&x_i^*(x_i)=1,\ x_i^*(-x_i)=-1,\ x_i^*(x_i+x_i)=2\text{ and }x_i^*(x_k)=0,\tag{1}\\
&\text{for every }1\leq k\neq i\leq n.\text{ So }x_i\text{ and }-x_i\text{are separated by }x_i^*.\\
&\tag{2}\phi_1(x)=\left\{\begin{array}{rl}
t, & x=x_i\\
-t, & x=x_j\\
-t, & x=-x_i\\
t, & x=-x_j\\
0, & x=\pm x_k,\ k\not\in\{i,j\}\\
t+t>1, & x=x_i-x_j\\
\end{array}\right.\\
&\text{So the pairs $x_i,x_j$ and $-x_i,-x_j$ are separated by $\phi_1$.}\\
&\tag{3}\phi_2(x)=\left\{\begin{array}{rl}
-s, & x=-x_i\\
s, & x=x_j\\
s, & x=x_i\\
-s, & x=-x_j\\
0, & x=\pm x_k,\ k\not\in\{i,j\}\\
s+s>1, & x=x_j-x_i\\
\end{array}\right.\\
\end{align*}
\reqnomode
Consequently the pairs $-x_i,x_j$, and $x_i,-x_j$, are separated by $\phi_2$, so we are finished.
\end{proof}
We mention here that for every uniformly smooth Banach space, $X$ and each $\{(x_i,x_i^*):i\in\N\}$ Auerbach system in $X$ there exists an 
infinite subset $M$ of $\N$ such that the set $\{\pm x_i:i\in M\}$ is a $(1,1,d)-$subset of $X$, with $d>1$. Indeed as in the proof of 
Proposition 1.4 we can prove that there exist an infinite subset $M$ of $\N$ and $\mu_{ij}\in(1,2]$, for $i\neq j\in M$ such that 
\[\norm{\mu_{ij}\left(\frac{1}{2}x_i^*+\frac{1}{2}x_j^*\right)}=1\]
and
\[\mu_{ij}\geq r>1\text{, for }i\neq j\in M.\]
Now as in the proof of Proposition 1.6 we take that $x_i,-x_i$ are separated by $x_i^*$ ,$x_i,x_j$ are separated 
by $\lambda_{ij}\left(\frac{1}{2}x_i^*-\frac{1}{2}x_j^*\right)$ and $x_i,-x_j$ are separated by $\mu_{ij}\left(\frac{1}{2}x_i^*+\frac{1}{2}x_j^*\right)$, for $i\neq j\in M$.
 Analogously it can be proved that for every Banach space $X$ such that $\dnormed$ is strictly convex and each Auerbach system in, $\{(x_i,x_i^*):i\in\N\}$ in $X$ the set $\{\pm x_i:i\in\N\}$ is a bounded and separated antipodal subset of $X$ with constants $c_1=c_2=1$ and $d=(1+)$.

It has been noted that the fact that, $K_a(X)>1$, for every Banach space $X$ is a strengthening of 
Elton-Odell's theorem. The proof of this fact in the case $X$ contains isomorphicaly $c_0$ or $l_p$ for some $1\leq p<\infty$
or $X$ is uniformly smooth is independent from the theorem of Elton-Odell. The last no 
longer remains true, in our approach, for the general case or the cases $X$ is reflexive or $X$ 
has separable dual. For the cases $X$ is reflexive and $X$ has separable dual are needed particular 
parts of the proof of Elton-Odell's theorem. For the general case Theorem 4.1 of [9], which is an 
independent proof of the theorem of Elton-Odell, is needed. To utilize these results 
we are going to prove Theorem 1.7 below. In our proof we need Theorem 1 of [17] which we now state.
\begin{unthm}[1 of {[17]}]
Let $\normed$ be a Banach space. Then for every $\epsilon>0$ there exists an equivalent norm $\norm{\cdot}'$ on $X$ such that:
\begin{enumerate}[label=(\arabic*)]
\item$d(\normed,(X,\norm{\cdot}'))\leq1+\epsilon$ (the Banach-Mazur distance) and
\item$(X,\norm{\cdot}')$ contains an infinite equilateral set.
\end{enumerate}
\end{unthm}
\begin{thm}
Let $X$ be a Banach space and $(x_n)$ a normalized weakly null and $d-$separated sequence in $X$ with $d>0$. Then $K_a(X)\geq d$.
\end{thm}
\begin{proof}
As $d-$separated, with $d>0$, the sequence $(x_n)$, has no norm convergent subsequence.
In addition $(x_n)$ is weakly null. From the method of the proof of Theorem 1 of [17] we have the following:

For every $\epsilon\in(0,1)$, there exist a subsequence of $(x_n)$, still denoted by $(x_n)$, and 
an equivalent norm $\norm{\cdot}'$ on $X$ such that
\leqnomode
\begin{gather}
\frac{1}{1+\epsilon}\norm{x}\leq\norm{x}'\leq\frac{1}{(1-\epsilon)^2}\norm{x}\text{, for every }x\in X\text{ and}\\\setcounter{equation}{2}
\norm{sx_n+tx_m}'=\lim_{\substack{k<l\\k\to\infty}}\norm{sx_k+tx_l}\text{, for every }n\neq m\text{ and }s,t\in\mathbb{R}.
\end{gather}
\reqnomode
By (1.3), $\norm{x_n}'=1$, $n\geq1$ and $\norm{x_n-x_m}'=\lim_{\substack{k<l\\k\to\infty}}\norm{x_k-x_l}=\lambda\geq d$, $n\neq m$. 
Thus $(x_n)$ is a $\lambda-$equilateral subset of $S_X^{\norm{\cdot}'}$ with $\lambda\geq d$. So from (2) 
of Remark 1.1 the set $\{x_n:n\in\N\}$ is a $(1,1,\lambda)-$b.s.a subset of $(X,\norm{\cdot}')$ (see also the proof of Proposition 2 of [17]). 
Now for $n\neq m\in\N$, we choose $f_{nm}$ in $S_{X^*}^{\norm{\cdot}'}$ such that,
\leqnomode
\begin{gather}
f_{nm}(x_n-x_m)=\norm{x_n-x_m}'=\lambda\geq d\text{ and}\\
f_{nm}(x_m)\leq f_{nm}(x_k)\leq f_{nm}(x_n)\text{, for every }k\not\in\{n,m\}.
\end{gather}
By (1.2) we also have for the dual norms
\begin{gather}
\frac{1}{1+\epsilon}\norm{x^*}'\leq\norm{x^*}\leq\frac{1}{(1-\epsilon)^2}\norm{x^*}'\text{, for every }x^*\in X^*.
\end{gather}
Now we set $g_{nm}=(1-\epsilon)^2f_{nm}$, then (1.6) gives $\norm{g_{nm}}\leq\norm{f_{nm}}'=1$, $n\neq m$. Moreover by (1.4)
\[g_{nm}(x_n-x_m)=(1-\epsilon)^2\lambda\geq(1-\epsilon)^2d,\ n\neq m.\]
Combining these inequalities with (1.5) we get that the set $\{x_n:n\in\N\}$ is a $(1,1,(1-\epsilon)^2d)-$b.s.a subset of $\normed$. 
Since $\epsilon\in(0,1)$ was arbitrary, $K_a(X)\geq d$.
\end{proof}
By Theorem 1.7 our strategy to prove Corollaries 1.8 and 1.9 and Theorem 1.10 will be to produce a normalized weakly null sequence which 
is $d-$separated with $d>1$. Also for the second and third of these results we assume, as we may by Proposition 1.2, that each Banach space does not contain isomorphicaly $c_0$ or $l_1$.
\begin{cor}
Let $X$ be a reflexive Banach space, then $K_a(X)>1$.
\end{cor}
\begin{proof}
Since the Banach space $X$ does not contain isomorphicaly $c_0$ from the proof of Elton-Odell's theorem we take that $X$ 
contains a normalized, basic and $d-$separated sequence $(x_n)$ with $d>1$. Since the sequence $(x_n)$ is basic we have, by Lemma 1.6.1 of [1]and by 
the compactness of $(B_X,w)$ that $0$ is the only weak 
cluster point of $(x_n)$ and hence $x_n\xlongrightarrow{w}0$.
\end{proof}
\begin{cor}
Let $X$ be a Banach space with separable dual, then $K_a(X)>1$.
\end{cor}
\begin{proof}
Since $X$ has separable dual we can choose a normalized and weakly null shrinking basic sequence  $(y_n)$
in $X$ (Prop. 1.b.13 [14] and Prop. 3.2.7 [1])). Further by the remarks that follow Theorem 1.a.5 of [14] we 
may also assume that $\norm{P_n}\leq1+20^{-n}$, $n\in\N$, for the associated projections to the basic sequence $(y_n)$. Now again by the 
proof of Elton-Odell's theorem (recall that $X$ does not contain isomorphically $c_0$) there exists a block 
basic sequence $(x_n)$ of $(y_n)$ and $d>1$ such that $\norm{x_n-x_m}>d$, for every $n\neq m\in\N$. Again by Proposition 3.2.7 of [1] the 
sequence $(x_n)$ is weakly null, so we are finished.
\end{proof}
Now we pass to the proof of the general case.
\begin{thm}
Let $X$ be a Banach space, then $K_a(X)>1$.
\end{thm}
\begin{proof}
Since $X$ does not contain isomorphicaly $l_1$ by Rosenthal's $l_1-$Theorem [19], we may choose a basic, normalized and 
weakly null sequence $(x_n)$ in $X$. Now $X$ does not contain isomorphicaly $c_0$, so by Theorem 4.1 of [9] there 
exists a normalized weakly null block-basic sequence $(y_n)$ of $(x_n)$ with spreading model $(e_i)$ such that $\norm{e_1-e_2}>1$. 
For the definition of spreading models we refer the reader to [3] and [1]. For our purposes suffices the following property
\[\lim_{\substack{n<m\\n\to\infty}}\norm{y_n-y_m}=\norm{e_1-e_2}>1.\]
\end{proof}
It has been noted that $K_a(X)\leq K(X)$, for any Banach space $X$. It is unknown for us if there exist a Banach space $X$ such that $K_a(X)<K(X)$. 
What we do have is a partial answer in the case $X$ is reflexive, where $K_a(X)=K(X)$.
\begin{thm}
Let $X$ be a reflexive Banach space, then $K_a(X)=K(X)$.
\end{thm}
\begin{proof}
Let $0<\lambda<K(X)$ and $(x_n)$ a normalized $\lambda-$separated sequence in $X$. By the reflexivity of $X$ we may assume 
that $x_n\xlongrightarrow{w}x_0$, for some $x_0\in B_X$. We consider now the seminormalized, 
weakly null and $\lambda-$separated sequence $(y_n)$ with $y_n=x_n-x_0$, $n\in\N$. As in Theorem 1.7 for every $\epsilon\in(0,1)$ 
there exists an equivalent norm $\norm{\cdot}'$ in $X$ such that,
\begin{gather*}
\frac{1}{1+\epsilon}\norm{x}\leq\norm{x}'\leq\frac{1}{(1-\epsilon)^2}\norm{x}\text{, for }x\in X,\\
\norm{sy_n+ty_m}'=\lim_{k<l}\norm{sy_k+ty_l}\text{, for }n\neq m,s,t\in\mathbb{R}\text{ and}\\
\frac{1}{1+\epsilon}\norm{x^*}'\leq\norm{x^*}\leq\frac{1}{(1-\epsilon)^2}\norm{x^*}'\text{, for }x^*\in X^*.
\end{gather*}
Thus the sequence $(y_n)$ is $\norm{\cdot}'-$equilateral, so the set $\{y_n:n\in\N\}$ is a $(c_1,1,\lambda)-$b.s.a subset in $\norm{\cdot}'$, 
with $c_1=\sup\{\norm{y_n}:n\in\N\}$. For $n\neq m\in\N$ we consider $f_{nm}\in B_{X^*}^{\norm{\cdot}'}$ such 
that $f_{nm}(y_m)\leq f_{nm}(y_k)\leq f_{nm}(y_n)$, for every $k\in\N$ and $f_{nm}(y_n)-f_{nm}(y_m)\geq\lambda$. Further we 
put $g_{nm}=(1-\epsilon)^2f_{nm}$, for $n\neq m\in\N$. Then for $n\neq m\in\N$ we have
\leqnomode
\begin{gather}
\norm{g_{nm}}\leq1\\
\begin{align}
g_{nm}(y_m)\leq g_{nm}(y_k)&\leq g_{nm}(y_n)\text{, for every }k\in\N\text{ and}\\\notag
g_{nm}(x_n-x_m)&=g_{nm}((x_n-x_0)-(x_m-x_0))\\\notag
&=g_{nm}(y_n-y_m)\\\notag
&=(1-\epsilon)^2f_{nm}(y_n-y_m)\\\notag
&\geq(1-\epsilon)^2\lambda.
\end{align}
\end{gather}
\reqnomode
Thus $K_a(X)\geq(1-\epsilon)^2\lambda$, for every $0<\epsilon<1$ and $\lambda<K(X)$, and cosequently $K_a(X)=K(X)$.
\end{proof}
\begin{rmrk}
Let us summarize some of the current knowledge known results concerning the parameters $K_a(X)$ and $K(X)$. Let X be a Banach space.
\begin{enumerate}[label=(\arabic*)]
\item$K_a(X)\leq K(X)\leq2$.
\item$K(Y)\leq K(X)$ and $K_a(Y)\leq K_a(X)$, where $Y$ is an infinite dimensional subspace of $X$.
\item$1<K(X)$ ([8]).
\item$\sqrt[5]{4}\leq K(X)$, if $X$ is non reflexive space ([13]).
\item$1<K_a(X)$ (Theorem 1.10).
\item$K_a(X)=K(X)$, if $X$ is reflexive (Theorem 1.11).

Particularly, $K_a(X)=K(X)\geq2^{1/p}$ if $X\cong l_p$, $1<p<\infty$ and $K_a(X)=K(X)=2$ if $X\cong c_0$ or $l_1$ 
(Proposition 1.2). Note that the case $1<p<\infty$ of Proposition 1.2 is an easy concequence of Theorem 3 of [12] and of Theorem 1.11.
\end{enumerate}
\end{rmrk}
Let $X$ be a Banach space. With $[X]$ we denote the class of Banach 
spaces $Y$ such that $X\cong Y$ and we define a pseudometric $D$ on $[X]$ in the following way:
\[D(X,Y)=\inf\{\log\norm{T}\norm{T^{-1}}:X\cong^TY\}.\]
In [12] Kottman defines for each Banach space $X$ the set
\[\overline{K(X)}=\{K(Y):Y\cong X\}\]
and proves (Theorem 7) that there exists $b\in[1,2]$ such that $\overline{K(X)}=(b,2]$ or $[b,2]$. The same way we define the set
\[\overline{K_a(X)}=\{K_a(Y):Y\cong X\}\]
and we prove a similar result.
\begin{prop}
For each Banach space $X$ there exists $b\in[1,2]$ such that $\overline{K_a(X)}=(b,2]$ or $[b,2]$.
\end{prop}
\begin{proof}
Firstly we will show that $\max\overline{K_a(X)}=2$. This can be done the same way that Kottman proved in Theorem 7 of [12] that $\max\overline{K(X)}=2$
(essentially Kottman proves that $\max\overline{K_a(X)}=2$). We briefly describe his argument. We consider 
an Auerbach system $\{(x_i,x_i^*):i\in\N\}$ in $X$ and put $V=\conv\{B_X\cup\{\pm2x_i:i\in\N\}\}$. Further we consider the Minkowski 
functional of $V$, which here is an equivalent norm, $\norm{\cdot}'$, on $X$ such that
\[\norm{x}\leq2\norm{x}'\leq2\norm{x},\ x\in X\]
(See also the proof of Theorem 3 of [17]). Now it is not difficult to prove 
that the set $\{\pm2x_i:i\in\N\}$ is a normalized 2-equilateral in the norm $\norm{\cdot}'$. Therefore $K_a(\norm{\cdot}')=2$.

We know that the topological space $([X],D)$ is connected ([12]), consequently it suffices 
to show that the function $K_a:[X]\to\mathbb{R}$ is continuous. Let $Y\in[X]$ and $\epsilon>0$. We choose $\delta>1$ such that
\begin{equation}
2-\frac{2}{\delta}<\epsilon\text{ and }\delta^2<K_a(Y)\ (K_a(Y)>1\text{, Theorem 1.10}).
\end{equation}
Let also $Z\in[X]$ with $D(Y,Z)\leq\log\delta$, then there exists a linear operator $T:Y\to Z$ with $\norm{y}_Y\leq\norm{T(y)}_Z\leq\delta\norm{y}_Y$, 
for every $y\in Y$ (so we have in addition that $\norm{T^{-1}}\leq1$). 
We consider an arbitrary $s$ with $\delta<s<K_a(Y)$, $\{y_i:i\in\N\}$ a $(1,1,s)-$b.s.a. subset 
of $Y$ and $f_{ij}\in B_{Y^*}$ such that $f_{ij}(y_j)\leq f_{ij}(y_k)\leq f_{ij}(y_i)$ and $f_{ij}(y_i)-f_{ij}(y_j)\geq s$, 
for every $i\neq j\in\N$ and $k\in\N$. Now we put
\[z_i=\frac{T(y_i)}{\delta}\text{ and }g_{ij}=f_{ij}\circ T^{-1}\text{, for }i\neq j\in\N.\]
Then $\norm{z_i}\leq1$ and $\norm{g_{ij}}\leq\norm{f_{ij}}\norm{T^{-1}}\leq1$, for $i\neq j\in\N$. Also if $k\in\N$
\begin{align*}
g_{ij}(z_j)&=\pn{f_{ij}\circ T^{-1}}\pn{\frac{T(y_j)}{\delta}}=f_{ij}\pn{\frac{y_j}{\delta}}\leq f_{ij}\pn{\frac{y_k}{\delta}}\\
&=g_{ij}(z_k)\leq f_{ij}\pn{\frac{y_i}{\delta}}=g_{ij}(z_i)
\end{align*}
and
\begin{align*}
g_{ij}(z_i-z_j)&=\pn{f_{ij}\circ T^{-1}}\pn{\frac{T(y_i)}{\delta}-\frac{T(y_j)}{\delta}}=\frac{1}{\delta}\pn{f\circ T^{-1}}\pn{T(y_i)-T(y_j)}\\
&=\frac{1}{\delta}f_{ij}(y_i-y_j)\geq\frac{s}{\delta}>1.
\end{align*}
Since $s$ was an arbitrary point of $\pn{\delta,K_a(Y)}$ we get that $K_a(Z)\geq\frac{K_a(Y)}{\delta}$. We have that
\[K_a(Y)-K_a(Z)\leq K_a(Y)-\frac{K_a(Y)}{\delta}\leq2-\frac{2}{\delta}<\epsilon\text{ (by 1.9)}.\]
Again by (1.9) $K_a(Y)\geq\delta^2$, so $K_a(Z)\geq\frac{K_a(Y)}{\delta}>\delta$. By the last inequality we can repeat our proof exchanging the 
roles of $Y$ and $Z$, to take $K_a(Z)-K_a(Y)<\epsilon$ thus $|K_a(Z)-K_a(Y)|<\epsilon$, so we are finished.
\end{proof}
\begin{rmrk}
Let $X$ be a Banach space.
\begin{enumerate}[label=(\arabic*)]
\item Let $Y$ be an infinite dimensional closed subspace of $X$. Since the sets $\overline{K(X)}$ and $\overline{K_a(X)}$ are 
intervals of the form $(b,2]$ or $[b,2]$ for some $b\in[1,2]$, by (2) of Remark 1.12 we take that
\[\overline{K(X)}\subseteq\overline{K(Y)}\text{ and }\overline{K_a(X)}\subseteq\overline{K_a(Y)}.\]
\item Proposition 1.2, Remark 1.3 and Remark 1.12 yield that:
\begin{enumerate}[label=(\alph*)]
\item If $X\cong l_p$, $1<p<\infty$, then $\overline{K_a(X)}=\overline{K(X)}=[2^{1/p},2]$.
\item If $X\cong l_1$ or $X\cong c_0$, then $\overline{K_a(X)}=\overline{K(X)}=\{2\}$.
\item If $X$ is reflexive, then $\overline{K_a(X)}=\overline{K(X)}$.
\item If $X$ is non reflexive, then $\inf\overline{K(X)}\geq\sqrt[5]{4}$ ([13]).
\item Elton-Odell Theorem ([8]) and Theorem 1.10 yield that
\[\overline{K(X)}\subseteq(1,2]\text{ and }\overline{K_a(X)}\subseteq(1,2].\]
\end{enumerate}
\end{enumerate}
\end{rmrk}
We note that the following question seems to be open:
\begin{adjustwidth}{1cm}{0cm}
Does there exist a Banach space $X$ such that $\overline{K(X)}=(1,2]$?

Should there exists such a space, it must be a space that is reflexive and does not contain isomorphicaly 
any of the spaces $l_p$, ( Remark 1.14 (2) (a) and (d)). So such a space would have the properties of the 
space of Tsirelson.
\end{adjustwidth}
The study of bounded and separated antipodal sets with constants $c_1=c_2=1$ and $d>1$ is also interesting 
in finite dimensional spaces. In this case we would be interested about the cardinality of such 
sets. We note that the cardinality of a $(1,1,d>1)-$b.s.a set in a finite dimensional space $\normed$ is also 
finite. By a result of Danzer and Gr\"unbaum [5] the maximum cardinality of an antipodal set in $\mathbb{R}^n$ is $2^n$, so the 
cardinality of a $(1,1,d>1)-$b.s.a set in $\normed$, where $\dim X=n$, cannot exceed $2^n$. On the other hand a $(1,1,d>1)-$b.s.a set in $X$
is $d-$separated, thus known estimations of the cardinality of normalized $d-$separated subsets of $X$ with $d>1$, may play some role 
see [2] and [10]). In Proposition 1.6 we proved that that if $X$ is a $n-$dimensional smooth Banach space then $X$ contains 
a $(1,1,d>1)-$b.s.a set of cardinality $2n$. Note in this connection that results on $(1,1,d>1)-$b.s.a sets will be appeared elsewhere.
\section*{References}
\begin{enumerate}[label={[\arabic*]}]
\item F. Albiac and N. J. Kalton, Topics in Banach space theory, Graduate Texts in Mathematics, vol. 233, Springer, New York, 2006.
\item L. Arias-de-Reyna, K. Ball and R. Villa, Concentration of the distance in finite dimensional normed spaces. Mathematika 45 (1998), 245-252.
\item A. Brunel and L. Sucheston, On B convex Banach spaces, Math. Syst. Theory 7 (1979), 294-299.
\item J. A. C. Burlak, R. A. Rankin and A. P. Robertson, The packing of spheres in the space $l_p$ , Proc. Glasgow Math. Ass. 4 (1958), 22-25.
\item L. Danzer and B. Gr\"unbaum \"Uber zwei Probleme bez\"uglich konvexer K\"orper von P. Erd\H{o}s und von V. L. Klee, Math. Z. 79 (1962), 95-99.
\item M. M. Day, On the basis problem in normed spaces, Proc. Amer. Math. Soc. 13 (1962), 655-658.
\item J. Diestel, Sequences and Series in Banach spaces, Graduate Texts in Mathematics, vol. 92, Springer-Verlag, New York, 1984
\item J. Elton and E. Odell, The unit ball of every infinite-dimensional normed linear space contains a $(1+\epsilon)-$separated sequence, Colloq. Math. 44 (1981), no. 1, 105-109.
\item D. Freeman, E. Odell, B. Sari and B. Zheng, On spreading sequences and asymptotic structures, Trans. Amer. Math. Soc. 370 (2018), 6933-6953.
\item E.Glakousakis, S.K. Mercourakis, On the existence of 1-separated sequences on the unit ball of a finite dimensional Banach space, Mathematika, 61 (2015), 547-558.
\item P. Hajek, V. Montesinos Santalucia, J. Vanderwerff and V. Zizler, Biorthogonal systems in Banach spaces, CMS Books in Mathematics, vol. 26, Springer, New York, 2008.
\item C. Kottman, Subsets of the unit ball that are separated by more than one, Studia Mathematika 53 (1975), 15-27.
\item A. Kryczka and S. Prus, Separated sequences in nonreflexive Banach spaces, Proc. Amer. Soc. 129 (2001), 155-163.
\item J. Lindenstrauss and L. Tzafriri, Classical Banach spaces I, Springer-Verlag, Berlin, 1977.
\item E. Maluta and P. L. Papini, Estimates for Kottman's separation constant in reflexive Banach spaces, Colloq. Math. 117 (2009), 105-119.
\item R. E. Megginson, An introduction to Banach Space Theory, Graduate Texts in Mathematics, vol.183, Springer, New York, 1998.
\item S. K. Mercourakis and G. Vassiliadis, Equilateral sets in infinite dimensional Banach spaces, Proc. Amer. Math. Soc. 142, (2014), 205-212.
\item C. M. Petty, Equilateral sets in Minkowski spaces, Proc. Amer. Math. Soc. 29 (1971), 369-374.
\item H.P.Rosenthal, A characterization of Banach spaces containing $l_1$, Proc. Nat. Acad. Sci. U.S.A 71 (1974), 2411-2413.
\end{enumerate}
\vspace{0.5cm}
NATIONAL AND KAPODISTRIAN UNIVERSITY OF ATHENS, \\DEPARTMENT OF MATHEMATICS, \\PANEPISTIMIOUPOLIS, 15784 ATHENS, GREECE\\
\textit{E-mail address: e.glakousakis@gmail.com}\\ \\ \\
NATIONAL AND KAPODISTRIAN UNIVERSITY OF ATHENS, \\DEPARTMENT OF MATHEMATICS, \\PANEPISTIMIOUPOLIS, 15784 ATHENS, GREECE\\
\textit{E-mail address: smercour@math.uoa.gr}
\end{document}